\documentclass[11pt,leqno]{amsart}
\usepackage{amsmath}
\usepackage{amsfonts}
\usepackage{amssymb}
\usepackage{graphicx}
\usepackage{hyperref}
\usepackage{mathrsfs}
\usepackage{relsize}
\usepackage[margin=3.20cm]{geometry} 
\theoremstyle{plain}
\newtheorem{theorem}{Theorem}

\newtheorem{definition}[theorem]{Definition}
\newtheorem{lemma}[theorem]{Lemma}

\newtheorem*{theorem*}{Theorem} 
\newtheorem*{corollary*}{Corollary}
\newtheorem*{proposition*}{Proposition}
\newtheorem*{problem*}{Problem}
\newtheorem*{acknowledgement*}{Acknowledgement}

\newtheorem{theoremA}{Theorem} 

\newtheorem{lemmaA}[theoremA]{Lemma}

\theoremstyle{definition}
\newtheorem{example}[theorem]{Example}

\newtheorem*{fact}{Fact}

\newcommand{\vol}{\mathrm{vol}}

\newcommand{\inte}{\mathrm{int}}

\renewcommand{\div}{\operatorname{div}}

\newcommand{\Hess}{\operatorname{Hess}}

\newcommand{\Lip}{\mathrm{Lip}}

\newcommand{\rr}{\mathbb{R}}
\renewcommand{\ss}{\mathbb{S}}

\newcommand{\nn}{\mathbb{N}}

\newcommand{\bb}{\mathbb{B}}

\newcommand{\ric}{\operatorname{Ric}}

\newcommand{\supp}{\operatorname{supp}}

\newcommand{\<}{\langle}
\renewcommand{\>}{\rangle}

\newcommand{\dist}{\mathrm{dist}}

\renewcommand{\d}{\delta}

\newcommand{\vp}{\varphi}

\newcommand{\GS}{\Sigma}
\newcommand{\GO}{\Omega}
\newcommand{\GG}{\Gamma}

\newcommand{\CC}{\mathcal{C}}
\newcommand{\CD}{\mathcal{D}}
\newcommand{\CE}{\mathcal{E}}
\newcommand{\CF}{\mathcal{F}}

\newcommand{\CK}{\mathcal{K}}
\newcommand{\CN}{\mathcal{N}}
\newcommand{\CS}{\mathcal{S}}
\newcommand{\CO}{\mathcal{O}}
\newcommand{\CT}{\mathcal{T}}

\newcommand{\C}{\mathscr{C}}

\renewcommand{\L}{\mathscr{L}}
\renewcommand{\P}{\mathscr{P}}

\newcommand{\U}{\mathscr{U}}

\newcommand{\W}{\mathscr{W}}

\begin{document}


\title[The Frankel property for self-shrinkers via elliptic PDE's]{The Frankel property for self-shrinkers from the viewpoint of elliptic PDE's}

\begin{abstract}
 We show that two properly embedded self-shrinkers in Euclidean space that are sufficiently separated at infinity must intersect at a finite point. The proof is based on a localized version of the Reilly formula applied to a suitable  $f$-harmonic function with controlled gradient. In the  immersed case, a new direct proof of the generalized half-space property is also presented.
\end{abstract}

\subjclass[2010]{53C42, 53C21}
%

\author[Debora Impera]{Debora Impera}
\address[Debora Impera]{Dipartimento di Scienze Matematiche "Giuseppe Luigi Lagrange", Politecnico di Torino, Corso Duca degli Abruzzi 24, Torino, Italy, I-10129}
\email{debora.impera@polito.it}

\author[Stefano Pigola]{Stefano Pigola}
\address[Stefano Pigola]{Universit\`a degli Studi di Milano-Bicocca\\ Dipartimento di Matematica e Applicazioni \\ Via Cozzi 55, 20126 Milano - ITALY}
\email{stefano.pigola@unimib.it}

\author[Michele Rimoldi]{Michele Rimoldi}
\address[Michele Rimoldi]{Dipartimento di Scienze Matematiche "Giuseppe Luigi Lagrange", Politecnico di Torino, Corso Duca degli Abruzzi 24, Torino, Italy, I-10129}
\email{michele.rimoldi@polito.it}

\maketitle



\section{Basic notation and purpose of the paper}

\subsection{Weighted manifolds}
By a {\it weighted manifold}  (also called {\it manifold with density}, or {\it smooth metric measure space}) we mean a triad
\[
M_{f} = (M,g,dv_{f})
\]
where $(M,g)$ is an $m$-dimensional Riemannian manifold with volume element $dv$, $dv_{f} = e^{-f}dv$ and $f:M \to \rr$ is a smooth weight function. The (obviously intrisic) geometric analysis of $M_{f}$ is related to bounds on its {\it Bakry-\'Emery Ricci curvature}
\[
\ric_{f} :=  \ric +  \Hess(f)
\]
combined with the analysis of its {\it weighted Laplacian}
\[
\Delta_{f} u = \div_{f} \nabla u  = \Delta u - g(\nabla f,\nabla u)
\]
where the {\it weighted divergence} is the operator
\[
\div_{f} X = e^{f} \div (e^{-f} X).
\]
An important example of weighted manifold is represented by the {\it Gaussian space}
\[
\rr^{m+1}_{f} = (\rr^{m+1} , \<\cdot,\cdot\>, e^{-\frac{1}{2}|x|^{2}}dx).
\]
Since the weight function is $f(x) = \frac{1}{2}|x|^{2}$ we have that the Bakry-\'Emery Ricci curvature is 
\[
 \ric^{\rr^{m+1}}_{f}  \equiv 1,
\]
hence $\rr^{m+1}_{f}$ is called a {\it shrinking Ricci soliton}. Moreover, the weighted Laplacian is the {\it Ornstein-Uhlenbeck} operator
\[
\Delta_{f} u = \Delta  u- \< \nabla u, x\>.
\]

\subsection{Self shrinkers of the MCF}
Given an isometrically immersed hypersurface in the weighted manifold $M^{m+1}_f$
\[
x : \GS^{m} \to M^{m+1}_{f},
\]
we introduce the corresponding {\it weighted mean curvature vector field} of the immersion as
\[
\mathbf{H}_{f} := \mathbf{H} + (\nabla f) ^{\perp},
\]
where we are using the convention $\mathbf{H}=\mathrm{tr}_{\Sigma}\mathbf{A}$, $\mathbf{A}$ being the vector-valued second fundamental form. Here $(\cdot)^{\bot}$ denotes the orthogonal projection on the normal bundle of $\GS$. We say that $x: \GS^{m} \to M^{m+1}_{f}$ is {\it $f$-minimal} if $\mathbf{H}_{f} \equiv 0$.\smallskip

A {\it self-shrinker} of the mean curvature flow (MCF) in the Euclidean space $\rr^{m+1}$  is an $f$-minimal hypersurface of the Gaussian space $\rr^{m+1}_{f}$. This is completely equivalent to require that the mean curvature vector field satisfies the equation
\begin{equation}\label{SS}
x^{\bot}=-\mathbf{H}.
\end{equation}

\subsection{Intrinsic {\it vs} extrinsic weighted structure}
Clearly, the self-shrinker $\GS^{m}$ inherits the weighted structure of the ambient space. Thus, intrinsically, we can consider the manifold with density
\[
\GS^{m}_{\tilde f} =(\GS^{m},g=x^{\ast}\<\cdot,\cdot\>,dv_{\tilde f})
\]
where $\tilde f = f \circ x$. It is customary to drop the ``tilde'' in the weight function and to write $\GS^{m}_{f}$. An important relation between the (intrinsic) Bakry-\'Emery Ricci tensor of $\GS^{m}_{f}$ and the extrinsic geometry of the $f$-minimal hypersurface $\GS^{m}$ comes from the Gauss equations. Indeed, it was observed in \cite{Ri1} that
\[
\ric^{\GS}_{f} \geq 1 - | \mathbf{A}|^{2}.
\]
As in the usual minimal surface theory, another important link between the intrinsic weighted geometry and the extrinsic properties of the self-shrinker comes from the $f$-Laplacian of the immersion. We have the following identity (see e.g. \cite{CoMi})
\begin{equation}
\Delta_{f}^{\Sigma}x = -x \label{Deltaf-immersion1}
\end{equation}
and its direct consequence
\begin{equation*}\label{Deltaf-immersion2}
\Delta_{f}^{\Sigma}|x|^2 = 2(m-|x|^2).
\end{equation*}

\subsection{Properly immersed self-shrinkers}
We are mainly interested in {\it properly immersed} self-shrinkers. A remarkable result by Q. Ding and Y.L. Xin, \cite{DiXi}, states that a properly immersed self-shrinker $x: \GS^{m} \to \rr^{m+1}_{f}$ has extrinsic  Euclidean volume growth 
\begin{equation}\label{polygrowth}
\left\vert \GS^{m} \cap \bb^{m+1}_{R} \right\vert= \CO(R^{m}) 
\end{equation}
and, hence, finite weighted volume
\begin{equation*}\label{finitevolf}
\vol_{f}(\GS^{m}) <+\infty.
\end{equation*}
This latter condition implies that the complete weighted manifold $\GS^{m}_{f}$ is {\it $f$-parabolic}, i.e., for any $u \in \C^{0}(\GS)\cap \W^{1,2}_{\mathrm{loc}}(\GS)$,
\begin{equation*}\label{fparab}
\begin{array}{ccc}
\begin{cases}
 \Delta_{f}^{\Sigma} u \geq 0 \\
 \sup_{\GS} u <+\infty 
\end{cases}
&
\Rightarrow u \equiv \mathrm{const.}
\end{array}
\end{equation*}
It is well known that parabolicity is a kind of compactness from several viewpoints, including global Stokes theorems and maximum principles, as it is already visible from the above definition.

\subsection{The Frankel property}
In many instances, properly immersed self-shrinkers behave like compact minimal hypersurfaces of the standard sphere. In this latter setting, it is well known that any two closed minimal immersed hypersurfaces must intersect. Actually, the ambient space can be generalized to a compact Riemannian manifold with strictly positive Ricci curvature. This is called the {\it Frankel property} after the celebrated paper by T. Frankel, \cite{Fr}. The original proof gives an estimate of  the (positive) distance between non-intersecting compact hypersurfaces in terms of their mean curvatures and it is based on the second variation of length along a  geodesic realizing the distance. New arguments and further extensions of the Frankel property to other geometric contexts are now available. Most notably, and relevantly for the development of the present paper, we mention \cite{PeWi} by P. Petersen and F. Wilhelm, where the maximum principle for the hypersurface-distance function is used, and 
\cite{FL} by A. Fraser and M. M.-C. Li, where the Frankel property is investigated in the setting of compact embedded free boundary minimal surfaces in a manifold with nonnegative Ricci curvature. We shall comment on these works later on, in Section \ref{section-controlled}. It is natural to ask:
\begin{problem*}
 Let $x_{j}:\GS^{m}_{j}\to \rr^{m+1}_{f}$, $j=1,2$, be complete, properly immersed self-shrinkers. To what extent is it true that $x_{1}(\GS^{m}_{1}) \cap x_{2}(\GS^{m}_{2}) \not=\emptyset$?
\end{problem*}
Starting from the work by G. Wei and W. Wylie, \cite{WeWy}, where the case of compact hypersurfaces is considered (actually what is really needed is that the positive distance between the hypersurfaces is realized at finite points), few partial positive answers to this question appeared in the literature. They are mostly related to (generalized) half-space properties of the shrinkers; see Section \ref{section-halfspace}.\smallskip

During the summer school ``Geometric Analysis on Riemannian and Singular Metric Measure Spaces'', held in Como in July 2016, \url{http://arms.lakecomoschool.org}, Prof. Tom Ilmanen suggested (and kindly outlined the main steps of the parabolic proof) that the Frankel property can be proved in the general framework of the motion by level-sets in Euclidean space.

\subsection{Purpose of the paper}
In this paper, by taking the purely elliptic viewpoint, we give the first general result in the literature about the validity of the (smooth) properly embedded Frankel property for self-shrinkers of the MCF. To this end, we shall collect results and techniques that, we feel, will be interesting also in other settings. The main contributions are the following:
\begin{itemize}
 \item [-]  With a new direct argument, based on the potential theory of weighted manifolds, we recover the main result of \cite{CaEs}, namely we show that a properly immersed self-shrinker cannot be located neither inside nor outside a self-shrinker cylinder; see Theorem \ref{TheoremExtIntCyl} in Section \ref{section-halfspace}.
 \item [-] We apply a localized version of the Reilly formula to a suitable  $f$-harmonic function with controlled gradient in order to show that two properly embedded self-shrinkers that are sufficiently separated at infinity must intersect at a finite point; see Theorem \ref{th-positivedistance} and Theorem \ref{theorem-intersection} in Section \ref{section-controlled}.
\end{itemize}

\section{Immersed shrinkers: half-space type properties}\label{section-halfspace}
After the celebrated paper by D. Hoffman and W. Meeks, \cite{HoMe}, one says that the {\it (weak) half-space property} holds for a certain family $\CF$ of immersed hypersurfaces, if any $\GS \in \CF$ cannot be confined in certain half-spaces unless it is a totally geodesic hyperplane.\smallskip

The first half-space property for Euclidean properly immersed self-shrinkers of the MCF was observed in \cite[Theorem 3]{PiRi}.
\begin{theorem*}
 Let $x: \GS^{m}\to\rr^{m+1}$ be a properly immersed self-shrinker. If $x(\GS)$ is contained in a closed half-space of $\rr^{m+1}$ determined by a hyperplane $\Pi$ passing through the origin, then $x(\GS) = \Pi$.
\end{theorem*}

In the same paper, the authors started the investigation on the possible regions where a properly immersed self-shrinker (with various geometric assumptions) can be located. The proof of the half-space property proposed in \cite{PiRi} is a simple application of the $f$-parabolicity of the self-shrinkers. Soon after, P. Cavalcante and J. Espinar, \cite[Theorem 1.1]{CaEs}, obtained the same result  using the touching principle, following closely the original proof by Hoffman-Meeks. The role of the catenoid is now played by a rotational self-shrinker discovered by S. Kleene and N. M\o ller \cite{KlMo}. With a different analytic technique, based on a perturbation argument that exploits the instability of the Jacobi operator associated to a cylinder, they were also able to replace the half-space by the interior or the exterior region of a cylindrical self-shrinker; \cite[Theorem 1.2 and Theorem 1.3]{CaEs}\smallskip

In the next result we use potential theoretic arguments to recover \cite[Theorem 1.2 and Theorem 1.3]{CaEs} in a very succinct way. 
\begin{theoremA}\label{TheoremExtIntCyl}
Let $x:\Sigma^m\to\mathbb{R}^{m+1}$ be a complete properly immersed self-shrinker. If $x(\Sigma)$ is  confined inside either one of the connected regions of $\rr^{m+1}$ determined by the self-shrinker cylinder $\mathbb{S}^{k}_{\sqrt{k}}\times\mathbb{R}^{m-k}\subset\mathbb{R}^{m+1}$, $1\leq k\leq m-1$, then $x(\GS) =\mathbb{S}^{k}_{\sqrt{k}}\times\mathbb{R}^{m-k}$. 
\end{theoremA}

\begin{proof}
Letting $\left\{e_{A}\right\}_{A=1}^{m+1}$ be an orthonormal frame of $\mathbb{R}^{m+1}$, we will denote the coordinate functions of $x$ by $x_A:=\left\langle x, e_{A}\right\rangle$ and by $\mathcal{N}$ the chosen (local) Gauss map of the immersion. Fix $1\leq k\leq m-1$ and consider a self-shrinker cylinder
\[
\mathcal{C}^{k}_{\sqrt{k}}:=\mathbb{S}^{k}_{\sqrt{k}}\times\mathbb{R}^{m-k}\subset\mathbb{R}^{m+1}.
\]
Let $u:\Sigma\to\mathbb{R}$ be the smooth function defined by $$u=\sum_{A=1}^{k+1}x_{A}^2.$$ Clearly, for every $p \in \GS$,
\[
u(p) =  \left( \overline{\dist}_{\rr^{m+1}}\left(x(p), \mathcal{C}^{k}_{\sqrt{k}}\right) + \sqrt{k} \right)^{2},
\]
where $\overline{\dist}_{\rr^{m+1}}\left(x(p), \CC^{k}_{\sqrt{k}}\right)$ denotes the signed distance. Here, we are using the convention that such a distance is negative inside $\CC^{k}_{\sqrt{k}}$.
Note that
\begin{align}\label{EqGradu}
\frac{1}{4}|\nabla^{\Sigma} u|^2 &=  u-\left(\sum_{A=1}^{k+1}x_A\left\langle e_A, \CN\right\rangle\right)^2\\ \nonumber
&= u-\left\langle\overline{x}, \overline{\CN}\right\rangle^2\\ \nonumber
&= u\left(1-\left\langle \frac{\overline{x}}{|\overline{x}|}, \overline{\CN}\right\rangle^2\right), \nonumber
\end{align}
where we are using the notation $\overline{x}=\sum_{A=1}^{k+1}x_{A}e_{A}$ and $\overline{\CN}=\sum_{A=1}^{k+1}\left\langle e_{A}, \CN\right\rangle e_{A}$. By Equation \eqref{Deltaf-immersion1}, we have that
\begin{eqnarray*}
\frac{1}{2}\Delta_{f}^{\Sigma}x_{A}^{2}&=& x_{A}\Delta_{f}^{\Sigma}x_{A}+ |\nabla^{\Sigma} x_{A}|^2\\
&=&- x_{A}^2+ |e_{A}^{T}|^2.
\end{eqnarray*}
Hence,
\begin{eqnarray}\label{EqfLaplu}
\frac{1}{2}\Delta_{f}^{\Sigma}u&=&  \sum_{A=1}^{k+1}|e_{A}^{T}|^2-u\\
&=& k+1-\sum_{A=1}^{k+1}\left\langle e_A, \CN\right\rangle^2-u \nonumber\\ 
&=& k+1 - |\overline \CN |^{2} -u. \nonumber
\end{eqnarray}
A direct consequence of \eqref{EqfLaplu} is that
\begin{align}\label{EqFLaplu1}
\frac{1}{2}\Delta_{f}^{\Sigma}u  \geq k-u.
\end{align}
On the other hand, using \eqref{EqGradu} and \eqref{EqfLaplu} we deduce the estimate 
\begin{align}
\Delta_{f}^{\Sigma}\sqrt{u}=&\frac{\Delta_{f}^{\Sigma}u}{2\sqrt{u}}-\frac{|\nabla^{\Sigma} u|^2}{4u^{\frac{3}{2}}}\label{EqfLaplsqrtu}\\
=&\frac{k-u}{\sqrt{u}}+\frac{\left\langle\frac{\overline{x}}{|\overline{x}|}, \overline{\CN} \right \rangle^2
-|\overline \CN |^{2}}{\sqrt{u}\nonumber}\\
\leq& \frac{k-u}{\sqrt{u}}\nonumber.
\end{align}

Assume now that $x(\Sigma)$ is confined in the closed exterior region determined by the self-shrinker cylinder $\CC^{k}_{\sqrt{k}}\subset\mathbb{R}^{m+1}$. Thus, $u \geq k$ and, by inequality \eqref{EqfLaplsqrtu}, $\sqrt{u}$ is $f$-superharmonic. On the other hand, since $\Sigma$ is properly immersed, $\Sigma_{f}$ is parabolic in the sense of \eqref{fparab}. It follows that $\sqrt{u} \equiv C$ for some constant $C \geq \sqrt{k}$ and using this information into \eqref{EqfLaplsqrtu} yields that, in fact, $u \equiv k$, i.e., $x(\GS) \subseteq \CC^{k}_{\sqrt{k}}$. The desired equality now follows by geodesic completeness.\smallskip

Similarly, suppose that $x(\Sigma)$ is confined inside the solid cylinder bounded by the self-shrinker $\CC^{k}_{\sqrt{k}}\subseteq\mathbb{R}^{m+1}$. Then, we see from \eqref{EqFLaplu1} that the function $u$ is $f$-subharmonic. Since $u \leq k$ it follows that $u$ must be constant by $f$-parabolicity. Whence, reasoning as in the previous case, we reach once more  the conclusion that $x(\GS) = \CC^{k}_{\sqrt{k}}$.
\end{proof}

\section{Embedded shrinkers with controlled asymptotic distance: main results}\label{section-controlled}

As we have alluded to in the introduction, a new approach to the traditional Frankel property that could be relevant in our setting was proposed by Petersen-Wilhelm in  \cite{PeWi}. As all of the other proofs of Frankel type properties, it proceeds by contradiction, assuming that two properly embedded minimal hypersurfaces $\GS_{1},\GS_{2}$ do not intersect. Then, one considers the function $u = r_{1}+r_{2}$ on the open set $\GO$ bounded by $\GS_{1}$ and $\GS_{2}$, where $r_{j}(x) = \dist(x,\GS_{j})$, $j=1,2$. At points $x$ where the normal exponential map $\exp_{\GS_{j}}^{\perp}$ is a diffeomorphism onto its image, and denoting with $\gamma_{j}$ the geodesic segment connecting $\GS_{j}$ with $x$, from the Riccati equation one obtains
\[
\frac{d}{dt} (\Delta r_{j} \circ\gamma_{j}(t)) = -1 - | \Hess(r_{j})|_{\gamma_{j}(t)} |^{2} \leq -1,
\]
which, once integrated, gives that $u$ satisfies, in the barrier sense of Calabi,
\[
\Delta u \leq - u, \text{ on }\GO.
\]
Now, if $\dist(\GS_{1},\GS_{2})>0$ is attained at some pair of points $(p,q) \in \GS_{1}\times \GS_{2}$ a segment $\gamma_{1}=\gamma_{2}$ connecting $p$ and $q$ is a curve of interior minima for $u$, contradicting the maximum principle. Similar arguments can be used when $\GS_{1},\GS_{2}$ are compact $f$-minimal hypersurfaces in the Gaussian space up to replacing the Laplace-Beltrami operator with the corresponding weighted Laplacian $\Delta_{f}$; see \cite[Theorem 7.3, Theorem 7.4]{WeWy}. However, if $\inf u$ is not achieved, i.e., $\dist(\GS_{1},\GS_{2})$ is either zero or realized at infinity, things are much more subtle. Since $\GO$ is an $f$-parabolic manifold with boundary, a natural attempt is to apply global forms of the maximum principle. Unfortunately, a formal computation suggests that the normal derivative of $u$ in the exterior direction  $\eta$ satisfies, in a suitable weak sense, $\partial_{\eta} u \leq 0$ on $\partial \GO$. Thus, it is unlikely that the potential theory for parabolic manifolds with boundary  developed in \cite{IPS} could be applicable directly to $u$ or to variants thereof.

In the  recent paper \cite{FL}, A. Fraser and M. Li proved that the Frankel property holds for compact embedded free boundary minimal surfaces in a compact manifold with non-negative Ricci curvature. They supplied two different arguments: the first one is a small variation of the original proof by Frankel whereas the second one, to the best of our knowledge, is completely new. It relies on the Reilly's formula applied to a harmonic function that separates the two surfaces. Inspired by this latter proof we shall obtain that if two properly embedded self-shrinkers are separated enough at infinity then they must intersect at some finite point. To make this claim more precise, we first observe that  two properly embedded shrinkers cannot be a positive distance apart.

\begin{theoremA}\label{th-positivedistance}
Let $\GS^{m}_{1},\GS^{m}_{2}$ be properly embedded $f$-minimal hypersurfaces inside the complete weighted manifold $M^{m+1}_{f}$ with $\ric_{f} \geq K^2>0$. Assume that each of $\GS_{1}$ and $\GS_{2}$ separates $M_{f}$. Then $\GS^{m}_{1}$ can not be a positive distance apart from $\GS^{m}_{2}$.
\end{theoremA}

We note that this result, in the setting of embedded hypersurfaces, extends \cite[Theorem 7.4]{WeWy} since it covers also the case where the two hypersurfaces realize their positive distance at infinity.
It remains to rule out the situation where $\GS_{1}$ and $\GS_{2}$ intersect at infinity. We achieve the goal in the  following assumptions:
\begin{itemize}
 \item {\it Asymptotic Distance Condition}: $\GS_{1}$ and $\GS_{2}$ are separated ``enough''  at infinity.\medskip
 \item {\it Tubular Neighborhood Condition}: One of the hypersurfaces, say $\GS_{2}$, is separated ``enough'' from itself at infinity.
\end{itemize}

In both cases, ``enough'' is quantified in exponential terms. More precisely:

\begin{theoremA}\label{theorem-intersection}
Let $\GS^{m}_{1}$ and $\GS^{m}_{2}$ be properly embedded connected self-shrinkers in the Euclidean space $\rr^{m+1}$. Assume that
\begin{equation}\label{AsymptHp}
\liminf_{\begin{array}{cc}|z|\to+\infty \\ z\in\GS_{2}\end{array}} \,\frac{\dist_{\mathbb{R}^{m+1}}(z, \Sigma_{1})}{e^{-b|z|^2} \P(|z|)^{-1}} >0,
\end{equation}
and that the ray $R(z)>0$ of a regular normal tubular neighborhood $\CT(\GS_{2})$ at a point $z \in \GS_{2}$ decays in the following controlled way
\begin{equation}\label{AsymptHp2}
 \liminf_{
\begin{array}{cc}
|z| \to +\infty \\ z \in \GS_{2} 
\end{array}
}
\frac{R(z)}{e^{-c|z|^{2}}\P(|z|)^{-1}} >0,
\end{equation}
for some polynomial $\P \in \rr[t]$ and some constants $b,c>0$ satisfying  $0<mc + b < \tfrac{1}{2}$. Then
\[
\Sigma_{1} \cap \Sigma_{2} \not=\emptyset.
\]
\end{theoremA}

 \begin{example}[$2$-dimensional]
Let us briefly illustrate a concrete situation where our results could be applied. Let $\GS_{1},\GS_{2} \hookrightarrow \rr^{3}$ be  properly embedded self-shrinkers with finite topology.\smallskip

By \cite{Wa-preprint}, the ends of each surface $\GS_{j}$ are smoothly asymptotic either to a cylinder or to a cone over a link in $\ss^{2}$. Examples with exactly one conical end are constructed by N. Kapouleas, S. Kleene and M. M\o ller, \cite{KKM}.\smallskip

Assume that each surface $\GS_{j}$ has only one end $\CE_{j}$, asymptotic to the cone $\CC_{j}$, $j=1,2$. Then, by \cite{Wa-JAMS, ChSc}, we know that $\CE_{j}$  has a uniform normal tubular neighborhood. Indeed:
 	\begin{itemize}
 	\item  [i)] $\CE_{j} = \operatorname{Graph}_{\CC_{j}}(w_{j})$ is a normal graph over $\CC_{j}$ with $|\nabla ^{k} w_{j}| \to 0$ at $\infty$, $\forall k\geq 0$. \smallskip
	\item [ii)] $\CE_{j}$ has unit normal $\nu_{j}$ asymptotic to the unit normal $\nu_{\CC_{j}}$ of $\CC_{j}$ at $\infty$. \smallskip
	\end{itemize}

Now, the following possibilities can occur:\medskip

\noindent (a) If $\CC_{1} = \CC_{2}$, then, by \cite{Wa-JAMS}, $\CE_{1}= \CE_{2}$ and in particular $\GS_{1} \cap \GS_{2} \not=\emptyset$.\smallskip

\noindent (b) If  $\CC_{1} \cap \CC_{2} = \{ 0\}$, i.e., the links of $\CC_{1}$ and $\CC_{2}$ are disjoint,  then $\CE_{1}$ and $\CE_{2}$ are very much separated and, hence, $\GS_{1}\cap \GS_{2} \not= \emptyset$.\smallskip

\noindent (c) If $\CC_{1}$ intersects transversally $\CC_{2}$ outside $\{0\}$, i.e. the links of $\CC_{1}$ and $\CC_{2}$ intersects transversally, then $\CE_{1} \cap \CE_{2} \not= \emptyset$ and in particular $\GS_{1} \cap \GS_{2} \not= \emptyset$.\smallskip

\noindent (d) If $\CC_{1}$ intersects $ \CC_{2}$ tangentially, i.e. the links of $\CC_{1}$ and $\CC_{2}$ are tangential, then we can still conclude that $\GS_{1}\cap\GS_{2} \not= \emptyset$ provided that  $\CE_{1},\CE_{2}$ are $e^{-b|z|^{2}}$-separated at $\infty$.
\end{example}

Both Theorem \ref{theorem-intersection} and Theorem \ref{th-positivedistance}, can be considered as concrete realizations of the following abstract result.

\begin{theoremA}\label{th-abstract}
 Let $\GS^{m}_{1},\GS^{m}_{2}$ be properly embedded $f$-minimal hypersurfaces inside the complete weighted manifold $M^{m+1}_{f}$ with $\ric_{f} \geq K^2>0$. Assume that $\GS_{1}\cap \GS_{2} =\emptyset$ and that each of $\GS_{1}$ and $\GS_{2}$ separates $M_{f}$. The domain enclosed by the two hypersurfaces is denoted by $\GO$. If $u \in \C^{\infty}(\bar \GO)$ is a solution of the following problem
\begin{equation}\label{absth-dirprob}
\begin{cases}
 \Delta_{f} u =0 & \mathrm{in}\,\,\GO\\
 u = 0 & \mathrm{on}\,\,\GS_{1}\\
 u = 1 & \mathrm{on}\,\,\GS_{2}
\end{cases}
\end{equation}
then
\begin{equation*}
\lim_{R \to +\infty} \frac{\int_{B^{M}_{R}\cap \GO} |\nabla u|^{2} dv_{f}}{R^{2}} = +\infty.
\end{equation*}
\end{theoremA}

The rest of the paper is organized as follows:
\begin{itemize}
 \item [-] First, in Section \ref{section-reilly}, we prove the abstract Theorem \ref{th-abstract} using a localized version of the Reilly's formula that, we feel, is of independent interest.
 \item [-] Next, in Section \ref{section-function}, we assume by contradiction that the two self-shrinkers do not intersect and we prove that a bounded solution of \eqref{absth-dirprob}, with finite Dirichlet $f$-energy, exists provided the assumptions of Theorem \ref{theorem-intersection} are met. This contradicts Theorem \ref{th-abstract}.
 \item [-] Finally, in Section \ref{section-variational}, using a variational viewpoint,  we still assume by contradiction that the two self-shrinkers do not intersect and we show that a bounded solution of \eqref{absth-dirprob}, with finite Dirichlet $f$-energy, exists without any restriction on the extrinsic geometry of $\GS_{2}$, provided the distance between the two self-shrinkers is strictly positive. This, again, contradicts Theorem \ref{th-abstract}.
\end{itemize}

\section{Localized Reilly's formula and a proof of Theorem \ref{th-abstract}}\label{section-reilly}
Reilly's formula is a celebrated integral formula first introduced in \cite{Re} to study isometric immersions and geometric bounds on the first Neumann eigenvalue of a compact manifold with boundary; see also \cite{ChWa} and the survey paper \cite{PiRiSe} for more applications. A version of this formula for the weighted Laplacian is obtained in \cite{MaDu} following the arguments in \cite{ChWa}. We are going to prove a localized version of the Reilly's formula that holds in the non-compact weighted setting.

\begin{lemma} Let $\Omega$ be a domain with smooth boundary $\GS = \partial \GO$ in the (possibly incomplete) smooth metric measure space $M^{m+1}_{f}$ and let $\phi\in \C_{c}^{\infty}(M)$. Then, for any $u\in \C^{2}(\overline \Omega)$, 
\begin{align}
&\int_{\Omega}\phi^2\left(|{\mathrm{Hess}}\,u|^2-\left({\Delta}_{f}u\right)^2+\mathrm{Ric}_{f}\left({\nabla} u, {\nabla}u\right)\right)dv_{f}+\int_{\Omega}\<{\nabla}\phi^2,\frac{1}{2}\nabla|\nabla u|^2 -{\Delta}_{f}u{\nabla}u\> dv_{f}\label{ReillyLoc}\\
&= \int_{\GS} \phi^2\left(\mathbf{A}_{\GS} (\nabla^{\GS} u, \nabla^{\GS} u) + \nabla^{\GS} u\left(\frac{\partial u}{\partial \nu}\right)-\left(\Delta^{\GS}_{f}u-H_{f}\frac{\partial u}{\partial \nu}\right)\frac{\partial u}{\partial \nu}\right)dv_{f}^{\Sigma},\nonumber
\end{align}
where $\nu$ is the exterior unit normal to $\Sigma$ and $\mathbf{A}_{\GS}$ is the corresponding second fundamental form, and $dv_{f}^{\Sigma}$ is the weighted measure of the boundary $\Sigma$.
\end{lemma}

\begin{proof}
Let $\phi\in C_{c}^{\infty}(M_{f})$ and let $Z={\nabla}\frac{|{\nabla}u|^2}{2}-{\Delta}_{f}u\,{\nabla}u$. Then, using the $f$-Bochner formula, we get
\begin{align*}
{\mathrm{div}}_{f}\left(\phi^2 Z\right)
=& \phi^{2}\left[{\mathrm{div}}_{f}\left({\nabla}\frac{|{\nabla} u|^2}{2}\right)-{\Delta}_{f}u\, {\mathrm{div}}_{f}({\nabla}u)-\left\langle{\nabla}{\Delta}_{f}u,{\nabla}u\right\rangle\right]\cr
&+\<{\nabla}\phi^2, {\nabla}\frac{|{\nabla}u|^2}{2}-{\Delta}_{f}u{\nabla}u\>\\
=&\phi^2\left[{\Delta}_{f}\frac{|{\nabla} u|^2}{2}-({\Delta}_{f}u)^2-\left\langle{\nabla}{\Delta}_{f}u, \nabla u\right\rangle\right]+ \<{\nabla}\phi^2, {\nabla}\frac{|{\nabla}u|^2}{2}-{\Delta}_{f}u{\nabla}u\>\\
=& \phi^2\left[|{\mathrm{Hess}}\,u|^2+\left\langle{\nabla}{\Delta}_{f}u, {\nabla} u\right\rangle + {\mathrm{Ric}}_{f}({\nabla} u, {\nabla} u)-({\Delta}_{f}u)^2-\left\langle{\nabla}{\Delta}_{f}u,{\nabla} u\right\rangle\right]\\
&+ \<{\nabla} \phi^2, {\nabla} \frac{|{\nabla} u|^2}{2}-{\Delta}_{f}u{\nabla} u\>\\
=&\phi^2\left[|{\mathrm{Hess}}\,u|^2+{\mathrm{Ric}}_{f}({\nabla} u, {\nabla} u)-({\Delta}_{f}u)^2\right]+\<{\nabla} \phi^2, {\nabla} \frac{|{\nabla} u|^2}{2}-{\Delta}_{f}u{\nabla} u\>.
\end{align*}
Applying the $f$-divergence theorem we get:
\begin{align}
\int_{\GS}\phi^2\<Z, \nu\>dv_{f}^{\Sigma}=&\int_{\Omega}{\mathrm{div}}_{f}(\phi^2Z)dv_{f}\label{ReillyLoc1}\\
=& \int_{\Omega} \phi^2\left[|{\mathrm{Hess}}\,u|^2+{\mathrm{Ric}}_{f}({\nabla} u, {\nabla} u)-({\Delta}_{f}u)^2\right]dv_{f} \cr
&+\int_{\GO} \<{\nabla} \phi^2, {\nabla} \frac{|{\nabla} u|^2}{2}-{\Delta}_{f}u{\nabla} u\>dv_{f}.\nonumber
\end{align}
Let $\left\{E_{i}\right\}_{i=1}^{m}$ be a local orthonormal frame on $\GS$. Letting $H_f:= H+\langle \nabla f,\nu\rangle$ and using the identity
\[
{\nabla}u=\nabla^{\GS} u+ \frac{\partial u}{\partial \nu}\nu,
\]
one gets that
\begin{align*}
{\Delta}_{f}u=&\sum_{i=1}^{m}{\mathrm{Hess}}\,u(E_{i}, E_{i})+{\mathrm{Hess}}\,u(\nu, \nu)-\<{\nabla} f, {\nabla} u\>\\
=&\Delta_{f}^{\GS}u+{\mathrm{Hess}}\,u(\nu, \nu)-H_{f}\frac{\partial u}{\partial \nu}
\end{align*}
and 
\begin{align*}
\<{\nabla}\frac{|{\nabla}u|^2}{2}, \nu \>=&{\mathrm{Hess}}\,u\left(\nabla^{\GS} u, \nu\right)+ \frac{\partial u}{\partial \nu}{\mathrm{Hess}}\,u(\nu, \nu)\\
=& - \<{\nabla}_{\nabla^{\GS} u}\nu, \nabla^{\GS} u \>+ \<{\nabla}_{\nabla^{\GS} u}\left(\frac{\partial u}{\partial \nu}\nu\right), \nu \> + \frac{\partial u}{\partial \nu}{\mathrm{Hess}}\,u(\nu, \nu)\\
=& \mathbf{A}_{\GS}(\nabla^{\GS} u, \nabla^{\GS} u) + {\nabla}^{\GS} u\left(\frac{\partial u}{\partial \nu}\right) + \frac{\partial u}{\partial \nu}\,{\mathrm{Hess}}\,u(\nu, \nu).
\end{align*}
Inserting these two relations in \eqref{ReillyLoc1} we obtain the desired conclusion.
\end{proof}

Theorem \ref{th-abstract} is a direct consequence of the localized Reilly's formula.
\begin{proof}[Proof of Theorem \ref{th-abstract}]
Let $u\in \C^{\infty}(\bar{\Omega})$ be a solution of 
\begin{equation*} 
\begin{cases}
 \Delta_{f} u= 0 & \text{in } \Omega\\
 u = 0 & \text{on } \GS_{1}\\
 u = 1 & \text{on } \GS_{2}
\end{cases}
\end{equation*}
Applying \eqref{ReillyLoc} to $u$ and using Young's  and Kato's inequalities we get that for any $\varepsilon>0$
\begin{align*}
0&\geq\int_{\Omega}\phi^2\left(|{\mathrm{Hess}}\,u|^2+K^2|{\nabla}u|^2\right)+ \<{\nabla}\phi^2,\frac{{\nabla}|{\nabla}u|^2}{2} \> dv_{f}\label{Ineq1}\\
&= \int_{\Omega}\phi^2\left(|{\mathrm{Hess}}\,u|^2+K^2|{\nabla}u|^2\right)+2\phi|{\nabla}u|\left\langle{\nabla}\phi,{\nabla}|{\nabla}u|\right\rangle dv_{f}\nonumber\\
&\geq \int_{\Omega}\phi^2\left(|{\mathrm{Hess}}\,u|^2+K^2|{\nabla}u|^2\right)- \frac{\phi^2}{\varepsilon}\left|{\nabla}|{\nabla}u|\right|^2-\varepsilon|{\nabla}\phi|^2|{\nabla}u|^2dv_{f}\nonumber\\
&\geq\int_{\Omega}\phi
^2\left(1-\frac{1}{\varepsilon}\right)\left|{\nabla}|{\nabla}u|\right|^2 dv_{f}+\int_{\Omega}\left(K^2\phi^2-\varepsilon|{\nabla}\phi|^2\right)|{\nabla}u|^2 dv_{f}.\nonumber
\end{align*}
Choose now $\varepsilon>1$ and let $\phi_{R}$ be smooth cut-offs such that $\phi_{R}=1$ on $B^{M}_{R}$, $\supp(\phi_{R})\subset B^{M}_{2R}$ and $|{\nabla} \phi|^2\leq \frac{4}{R^2}$. Then, we get
\begin{align*}
K^2\int_{\Omega\cap B^{M}_{R}}|{\nabla}u|^2dv_{f}&\leq K^2\int_{\Omega\cap B^{M}_{2R}}\phi^2|{\nabla} u|^2 dv_{f}\leq \varepsilon \int_{\Omega\cap B^{M}_{2R}}|{\nabla}\phi|^2|{\nabla} u|^2d v_{f}\\
&\leq \frac{4 \varepsilon}{R^2}\int_{\Omega\cap B^{M}_{2R}}|{\nabla}u|^2dv_{f}\\
\end{align*}
Taking the limit as $R\to+\infty$ this yields that 
\begin{equation}\label{limsup}
\int_{\GO} |\nabla u|^{2} dv_{f} \leq \frac{4\varepsilon}{K^2} \liminf_{R\to+\infty}\frac{1}{R^2} \int_{\Omega\cap B^{M}_{2R}}|{\nabla}u|^2dv_{f}.
\end{equation}
Since, obviously, $u$ is non-constant, we must have
\[
\liminf_{R\to+\infty}\frac{1}{R^2} \int_{\Omega\cap B^{M}_{2R}}|{\nabla}u|^2dv_{f} >0.
\]
In particular, $|\nabla  u | \not \in \L^{2}(\GO , dv_{f})$ and from inequality \eqref{limsup} we conclude
\[
\lim_{R\to+\infty}\frac{1}{R^2} \int_{\Omega\cap B^{M}_{2R}}|{\nabla}u|^2dv_{f}  = +\infty,
\]
thus completing the proof.
\end{proof}

\section{Construction of special \texorpdfstring{$f$}{f}-harmonic functions and proof of Theorem \ref{theorem-intersection}}\label{section-function}

The strategy of the proof of Theorem \ref{theorem-intersection} goes as follows:\smallskip

\begin{itemize}
 \item [-] By contradiction we assume that $\GS_{1}$ and $\GS_{2}$ are disjoint. Recall that, by the Jordan-Brouwer separation Theorem, each hypersurface $\GS_{j}$ separates $\rr^{m+1}$; \cite{GP}. Therefore it is well defined the region $\Omega$ of $\rr^{m+1}$ in between $\GS_{1}$ and $\GS_{2}$.\smallskip
 \item [-] We construct on $\Omega$ a (unique) bounded positive $f$-harmonic function $u$ with Dirichlet boundary conditions $0$ and $1$, respectively, on $\Sigma_{1}$ and $\Sigma_{2}$. To construct $u$, we solve a mixed boundary value problem along an exhaustion of $\Omega$ where a Neumann condition is imposed on the ``free'' part of the boundary. An application of interior and boundary Schauder estimates enables us to extract a subsequence converging  in $\C^{2}$ on every compact set  to the desired global solution.\smallskip
 \item [-] We use the asymptotic distance assumption \eqref{AsymptHp} and the condition on the extrinsic geometry of $\GS_{2}$ to obtain that, in fact, $|\nabla u| \in \L^{2}(\Omega,dv_{f})$. 
To this end, a Caccioppoli inequality up to the boundary reduces the problem to an estimate of $|\nabla u|$ on the hypersurface $\GS_{2}$. Thanks to the control on the extrinsic geometry of $\GS_{2}$, a pointwise estimate of $|\nabla u|$ on $\GS_{2}$ is obtained using maximum principle arguments.\smallskip
 \item [-] We use Theorem \ref{th-abstract} to get a contradiction.\smallskip
\end{itemize}

Accordingly, the main purpose of the present section is to prove the following

\begin{lemmaA}\label{lemma-fharmfunction}
Let $\GS^{m}_{1}$ and $\GS^{m}_{2}$ be disjoint, properly embedded hypersurfaces in the complete weighted manifold  $M^{m+1}_{f}$ satisfying $\ric_{f} \geq K^{2}>0$.  Assume that $\GS_{1},\GS_{2}$ separate $M$ and let $\Omega \subset M$ be the domain enclosed between these hypersurfaces so that $\partial \Omega = \GS_{1} \cup \GS_{2}$.  Then, there exists a unique bounded solution $u \in \C^{\infty}(\bar\GO)$ of the problem
\begin{equation} \label{DirichletProblem}
\begin{cases}
 \Delta_{f} u= 0 & \mathrm{in}\,\, \Omega\\
 u = 0 & \mathrm{on}\,\, \GS_{1}\\
 u = 1 & \mathrm{on}\,\, \GS_{2}
\end{cases}
\end{equation}
satisfying
\begin{equation*}
0 < u <1,\quad \mathrm{on}\,\, \GO.
\end{equation*}
Moreover, assume that $M^{m+1}_{f} = \rr^{m+1}_{f}$ is the Gaussian space and that $\Sigma_{2}$ satisfies the asymptotic distance condition  \eqref{AsymptHp} and the normal neighborhood condition \eqref{AsymptHp2}. Then $u$ has finite $f$-energy:
\begin{equation}\label{L1-integrability}
|\nabla u| \in \L^{2}(\Omega,dv_{f}).
\end{equation}
\end{lemmaA}
\smallskip

As we have outlined above, we split the proof in several steps.

\subsection{Uniqueness of the solution}\label{section-uniqueness}
Following \cite{IPS}, it is convenient to set the following
\begin{definition}
 Let $M_{f}$ be a smooth metric measure space with boundary $\partial M \not= \emptyset$. Say that $M_{f}$ is $f$-parabolic in the sense of Dirichlet if every bounded $f$-harmonic function $u \in \C^{\infty}(\inte M)\cap \C^{0}(M)$ is uniquely determined by its boundary values.
\end{definition}
 Adapting to the framework of manifolds with density what is known from \cite{IPS} (see also \cite{ILPS}), we have the following
\begin{lemma}
 Let $M_{f}$ be a complete weighted manifold with boundary such that, for some origin $o\in \inte M$, $
 \vol_{f}(B_{R}(o)) = \CO(R^{2})$. Then $M_{f}$ is $f$-parabolic in the sense of Dirichlet.
\end{lemma}
Since, in the setting of Lemma \ref{lemma-fharmfunction}, $\vol_{f}(\GO) \leq \vol_{f}(M)<+\infty$, the complete manifold $\bar \GO$ with boundary $\partial \GO \not= \emptyset$ is $f$-parabolic in the sense of Dirichlet. Hence a bounded solution of \eqref{DirichletProblem}, if any, must be unique.

\subsection{Existence of the solution by exhaustion}\label{section-existencesolution}

 As in the statement of Lemma \ref{lemma-fharmfunction}, let $M^{m+1}_{f}$ be a complete smooth metric measure space and let $\bar \GO \subset M$  be the domain whose boundary is given by $\partial \Omega = \GS_{1} \cup \GS_{2}$. Let $D_{k} \nearrow M$ be an exhaustion of $M$ by relatively compact domains with smooth boundary $\partial D_{k}$ intersecting transversally $\GS_{1}$ and $\GS_{2}$; see e.g. \cite{PePiSe}. Let $\GO_{k} \nearrow \GO$ be the Lipschitz domains defined by $\GO_{k} = D_{k} \cap \GO$. Note that,
\begin{equation*}
\partial \Omega_{k} = \GS_{1,k} \cup \GS_{2,k} \cup \GG_{k}
\end{equation*}
where $\GS_{i,k} \subset \GS_{i}$, $i=1,2$, and $\GG_{k} \subset \partial D_{k}$. The singular set of $\GO_{k}$ is denoted by $\CS_{k}$. Then, $\{\Omega_{k}\}$ is a ``good'' exhaustion of $\Omega$ with respect to mixed boundary value problems. Consider the solution $u_{k}$ to the problem
\begin{equation}\label{MixedBndPb}
 \begin{cases}
 \Delta_{f} u_{k} = 0 & \textrm{in }\Omega_{k} \\
 u_{k} = 0 & \textrm{on } \GS_{1,k}  \\
 u_{k} = 1 & \textrm{on } \GS_{2,k}  \\
\frac{ \partial u_{k}}{\partial \nu_{k}} = 0 & \text{on } \GG_{k}, \\
\end{cases}
\end{equation}
where $\nu_{k}$ is the outward unit normal to $\GG_{k}$. It follows from the Perron construction in \cite{Li1}, and the well-known local regularity theory, that $u_{k}\in \C^{0}(\bar{\Omega}_{k})\cap \C^{\infty}\left(\bar \Omega_{k}\setminus \CS_{k}\right)$. Furthermore, by the strong maximum principle and the boundary point lemma,
\begin{equation*}
0<u_{k}<1, \quad \text{on }\Omega_{k}.
\end{equation*}
For any fixed $k_{0} \in \nn$, let us consider the sequence of solutions of \eqref{MixedBndPb}:
\[
\U_{k_{0}+2} = \{u_{k} : k \geq k_{0}+2 \} \subset \C^{\infty}(\bar \GO_{k_{0}+1}).
\]
We claim that, given $\alpha>0$, there exists a constant $C_{k_{0}}>0$ such that
\begin{equation}\label{c2alpha-estimate}
\sup_{u_{k} \in \U_{k_{0}+2}}  \| u_{k} \|_{\C^{2,\alpha}(\bar\GO_{k_{0}})} \leq C_{k_{0}}.
\end{equation}
To this end:
\begin{itemize}
 \item [-] We apply \cite[Corollary 6.7]{GT} with the choices $\GO:= \GO_{k_{0}+1}$, $T:= \GS_{1,k_{0}+1}$, $L := \Delta_{f}$, $\varphi:= 0$, $f:= 0$ and we obtain that there exists a ray $\d_{1,k_{0}}>0$ and a constant $C_{1,k_{0}}>0$ such that, having defined the $\d_{k_{0}}$-neighborhood of $\GS_{1,k_{0}}$ as
\[
\CT_{\d_{1,k_{0}}}(\GS_{1,k_{0}}) = \cup_{p \in \GS_{1,k_{0}}} B^{M}_{\d_{1,k_{0}}}(p)\cap \GO_{k_{0}+1}
\]
it holds
\begin{equation*}
 \| u_{k} \|_{\C^{2,\alpha}\left(\CT_{\d_{1,k_{0}}}(\GS_{1,k_{0}}) \right)} \leq C_{1,k_{0}},
\end{equation*}
for every $u_{k} \in \U_{k_{0}+2}$.
\item [-] Similarly, we apply \cite[Corollary 6.7]{GT} with the choices $\GO:= \GO_{k_{0}+1}$, $T := \GS_{2,k_{0}+1}$, $L := \Delta_{f}$, $\varphi:= 1$, $f:= 0$ and we obtain that there exists a ray $\d_{2,k_{0}}>0$ and a constant $C_{2,k_{0}}>0$ such that
\begin{equation*}
 \| u_{k} \|_{\C^{2,\alpha}\left(\CT_{\d_{2,k_{0}}}(\GS_{2,k_{0}}) \right)} \leq C_{2,k_{0}},
\end{equation*}
for every $u_{k} \in \U_{k_{0}+2}$.
\item [-] Set
\[
\d_{k_{0}} = \min(\d_{1,k_{0}}, \d_{2,k_{0}}) >0
\]
and define the compact set
\[
\CK_{k_{0}} = \bar \GO_{k_{0}} \setminus \left( \CT_{\frac{\d_{k_{0}}}{2}}(\GS_{1,k_{0}}) \cup \CT_{\frac{\d_{k_{0}}}{2}}(\GS_{2,k_{0}})\right) \subset \GO_{k_{0}+1}.
\]
By \cite[Theorem 6.2]{GT} with $\GO := \GO_{k_{0}+1}$ and $L := \Delta_{f}$, there exists a constant $C_{3,k_{0}}>0$ such that
\begin{equation*}
  \| u_{k} \|_{\C^{2,\alpha}(\CK_{k_{0}})} \leq C_{3,k_{0}},
\end{equation*}
for every $u_{k} \in \U_{k_{0}+2}$.
\item [-] Since
\[
\GO_{k_{0}} \subset \CT_{\d_{k_{0}}}(\GS_{1,k_{0}}) \cup \CT_{\d_{k_{0}}}(\GS_{2,k_{0}}) \cup \CK_{k_{0}},
\]
and $\U_{k_{0}+2}\subset\C^{\infty}(\bar{\Omega}_{k_{0}+1})$, the claimed estimate \eqref{c2alpha-estimate} follows by taking $C_{k_{0}} = \max(C_{1,k_{0}},C_{2,k_{0}},C_{3,k_{0}})$.
\end{itemize}
Now,  for $0 < \alpha_{1} < \alpha_{2}$, the embedding $\C^{2,\alpha_{2}}(\bar \GO_{k_{0}}) \hookrightarrow \C^{2,\alpha_{1}}(\bar \GO_{k_{0}})$ is compact. Therefore, possibly passing to a subsequence, we obtain that $\U_{k_{0}+2}$ converges in $\C^{2}(\bar \Omega_{k_{0}})$ to a solution $u_{k_{0}} \in \C^{2}(\bar \GO_{k_{0}})$ (actually $u_{k_{0}} \in \C^{\infty}(\bar \GO_{k_{0}})$ by higher elliptic regularity) of the problem
\[
\begin{cases}
 \Delta_{f} u _{k_{0}} =0, & \text{in } \GO_{k_{0}}\\
 u_{k_{0}} = 0, & \text{on }\GS_{1,k_{0}}\\
 u_{k_{0}} = 1, & \text{on }\GS_{2,k_{1}}.
\end{cases}
\]
Moreover,
\[
0 \leq u_{k_{0}} \leq 1.
\]

To conclude the construction, we let $k_{0}$ increase to $+\infty$ and we use a classical diagonal argument. This yields the desired solution $u \in \C^{\infty}(\bar \GO)$ of 
\[
\begin{cases}
 \Delta_{f} u =0, & \text{in } \GO\\
 u = 0, & \text{on }\GS_{1}\\
 u = 1, & \text{on }\GS_{2},
\end{cases}
\]
satisfying
\[
0 \leq u \leq 1.
\]

\subsection{The finite \texorpdfstring{$f$}{f}-energy condition}

From now on, unless otherwise specified, we assume that $M_{f}$ is the Gaussian space $\mathbb{R}_{f}^{m+1}$. We want to obtain the finiteness of the Dirichlet $f$-energy of the solution $u$ of \eqref {DirichletProblem}:  $\| \nabla u \|_{\L^{2}(\Omega,dv_{f})}<+\infty$. A standard Caccioppoli inequality up to the boundary reduces the problem to an estimate of  $|\nabla u|$ along $\GS_{2}$ (the boundary hypersurface where the datum $1$ is imposed). This latter, in turn, can be carried out by maximum principle considerations.
\subsubsection{A Caccioppoli inequality up to the boundary}
We prove the following
\begin{lemma}\label{lemma-caccioppoli}
Let $u: \Omega \to \rr$ be a solution of \eqref{DirichletProblem}.
Then,
\begin{equation}\label{C2}
 \int_{\GO}|\nabla u |^{2} dv_{f}   \leq 2  \int_{\GS_{2}} |\nabla u| \, dv_{2;f_{2}} 
\end{equation}
where $dv_{i;f_{i}}$ is the weighted measure of the boundary hypersurface $\GS_{i,f_{i}}$ and $f_{i} = f|_{\Sigma_{i}}$, $i=1,2$.
\end{lemma}

\begin{proof}
 Let $\vp = \vp_{R}\in \C^{\infty}_{c}(\rr^{m+1})$ be a family of standard cut-off functions satisfying the following conditions:
 
\begin{enumerate}
 \item [i)] $0 \leq \vp \leq 1$;
 \item [ii)] $\vp = 1$ on $\bb^{m+1}_{R/4}(\bar x)$;
 \item [iii)] $\vp = 0$ on $\rr^{m+1} \setminus \bb^{m+1}_{R/2}(\bar x)$;
 \item [iv)] $\| \nabla \vp \|_{\infty} \leq 2/R$.
\end{enumerate}
Consider the vector field
\[
X = \vp^{2} u \nabla u
\]
Observe that the $f$-divergence of $X$ is given by
\[
\div_{f}X = 2u\vp \langle \nabla u , \nabla \vp\rangle + \vp^{2} |\nabla u|^{2}.
\]
Integrating $X$ on $\GO$, using the $f$-divergence theorem, the Young inequality and recalling the (Dirichlet) boundary conditions satisfied by $u$, we get
\begin{align*}
\int_{\GO} \vp^{2}|\nabla u|^{2} dv_{f} &= -\int_{\GO} 2u\vp \langle \nabla u , \nabla \vp\rangle dv_{f} + \int_{\GS_{2}} \vp^{2} \frac{\partial u}{\partial \nu_{2}} \, dv_{2;f_{2}}\\
&\leq \varepsilon \int_{\GO} u^{2} |\nabla \vp|^{2}dv_{f} + \varepsilon^{-1} \int_{\GO} \vp^{2} |\nabla u|^{2} \, dv_{f} +  \int_{\GS_{2}} \vp^{2}  \frac{\partial u}{\partial \nu_{2}} \, dv_{2;f_{2}}
\end{align*}
where $\varepsilon>0$ is any fixed constant. Whence, if $\varepsilon =2$ we deduce
\[
 \int_{\GO} \vp^{2}|\nabla u|^{2} dv_{f} \leq  2 \left\{2 \int_{\GO} u^{2} |\nabla \vp|^{2} dv_{f} +  \int_{\GS_{2}} \vp^{2}  \frac{\partial u}{\partial \nu_{2}} \, dv_{2;f_{2}}\right\}.
\]
Now, by the boundary point lemma and the maximum principle we have
\[
\GS_{2} = \{ u = 1\},\quad \nabla u \not = 0, \text{ on }\GS_{2}
\]
and the outward unit normal to the hypersurface $\GS_{2}$ is given by
\[
\nu_{2} = \frac{\nabla u}{|\nabla u|}.
\]
Inserting this information into the above integral inequality, and recalling property iv) of $\vp$, we deduce
\[
  \int_{\GO} \vp^{2}|\nabla u|^{2} dv_{f}  \leq \frac{16}{R^{2}} \int_{\GO \cap \bb^{m+1}_{2R}} u^{2} dv_{f} +  2 \int_{\GS_{2}} \vp^{2}  |\nabla u| \, dv_{2;f_{2}}.
\]
Whence, using once again the properties of $\vp = \vp_{R}$, the fact that $u$ is bounded, hence $\L^{2}(\GO,dv_{f})$, on the finite $f$-measure domain $\GO$, and letting $R \to +\infty$ we conclude the validity of \eqref{C2}.
\end{proof}


\subsubsection{Pointwise gradient estimates at the boundary}
It remains to estimate $\int_{\GS_{2}}|\nabla u| dv_{2;f_{2}}$. Since the self-shrinker $\Sigma_{2}$ is properly embedded, it has  polynomial (actually Euclidean) extrinsic volume growth. Assuming that $\GS_{2}$ has bounded extrinsic geometry, we can estimate $|\nabla u|$ pointwise along $\GS_{2}$ by  maximum principle arguments. This is done in the next Lemma. Whence, using \eqref{AsymptHp} we shall deduce immediately the desired $\L^{1}$ integrability of $|\nabla u|$ on $\GS_{2}$, thus completing the proof of Lemma \ref{lemma-fharmfunction}.
 
\begin{lemma}\label{lemma-boundarygradient}
 Let $u \in \C^{\infty}(\bar \GO)$ be a solution of \eqref{DirichletProblem}. Assume that $\GS_{2}$ satisfies an exterior $R$-sphere condition at some point $z \in \GS_{2}$. Then
 \begin{equation}\label{uniformgradest}
 |\nabla u|(z) \leq  \frac{(R+1)^{m}}{R^{m}} \frac{e^{|z|}}{\dist_{\rr^{m+1}}(z,\GS_{1})}.
\end{equation}
\end{lemma}

Recall that a connected component $\GS$ of the boundary $\partial \GO$ of a domain $\GO\subset \rr^{m+1}$ is said to satisfy the {\it exterior $R$-sphere condition at $z \in \GO$} if there exists a ball $\bb^{m+1}_{R}(y) \subset \rr^{m+1}\setminus \bar\GO$ such that $\bar \bb^{m+1}_{R}(y)$ is tangent to $\GS$ at $z$. Clearly the exterior $R$-sphere condition at $z$ implies that the (scalar) second fundamental form $A_{\GS} = \< \mathbf{A}_{\GS},\nu\>$ of $\GS$, with respect to the exterior unit normal $\nu$, satisfies $A_{\GS}(z) \leq 1/R$ in the sense of quadratic forms. More importantly, we point out the following simple fact that gives the link between the exterior sphere condition and the tubular neighborhood assumption in Lemma \ref{lemma-fharmfunction}.
\begin{fact}
The exterior $R$-sphere condition at $z \in \GS$ is implied by the existence of a regular normal neighborhood  of $\GS$ whose width at $z$ is at least $R$.
\end{fact}
\begin{proof}
Indeed, start with a small ball $ \bb^{m+1}_{a}(y)$ touching $\GS$ only at $z$ and let the ray $a>0$ increase by keeping the same tangent property. The corresponding balls are all centered along the normal line to $\GS$ passing through $z$. Let $T(z) \in \rr_{>0}\cup\{+\infty\}$ denote the supremum of such rays and assume $T(z)<+\infty$, for otherwise there is nothing to prove. Then $ \bb^{m+1}_{T(z)}(y')$ is tangent to $\GS$ at  two distinct points $z,z' \in \GS$. The (closed) segments normal to $\GS$, of length $2T(z)$ and centered  respectively in $z$ and $z'$ meet precisely at the center $y'$ of the ball. Therefore, by definition, the ray $R(z)$ of the normal regular tubular neighborhood of $\GS$ at $z$ must satisfy $R(z) \leq T(z)$. The conclusion now follows trivially.
\end{proof}

\begin{proof}[Proof (of Lemma \ref{lemma-boundarygradient})]
We use the technique developed in \cite{Tr} (see also \cite[Chapter 14]{GT}). Keeping the notation introduced in the previous section, we let
\[
v = u-1.
\]
Thus, $v$ is a smooth solution of
\begin{equation}\label{dir-probl-v}
 \begin{cases}
 \Delta_{f} v = 0 & \textrm{in } \Omega\\
 v= -1 & \textrm{on } \GS_{1}\\
 v=0 & \textrm{on } \GS_{2}\\
-1<v<0 & \textrm{in }\GO.
\end{cases}
\end{equation}
Fix $z \in \GS_{2}$. Let $\bb_{R}^{m+1}(y) \subset \rr^{m+1} \setminus \bar\GO$ be an exterior ball tangent to $\GS_{2}$ at $z$. We assume that
\[
\partial \bb^{m+1}_{R+a}(y) \cap \GS_{1} = \emptyset,
\]
for a suitable $a>0$. Namely, we choose
\[
0<a<\dist_{\rr^{m+1}}(z,\GS_{1}).
\]
Define the domain
\[
W_{R,a} = \left( \bb^{m+1}_{R+a}(y) \setminus \bar\bb^{m+1}_{R}(y) \right)\cap \GO.
\]
Let
\[
r(x) = |x-y| \quad \text{and}\quad d(x) = r(x) -R
\]
so that
\[
d(x) = \dist_{\rr^{m+1}}(x,\partial \bb^{m+1}_{R}(y)).
\]
We construct a smooth function
\[
\psi(d): [0,a] \to [0,+\infty)
\]
satisfying the following conditions\smallskip
\begin{itemize}
 \item [i)] $\Delta_{f}\psi (d(x)) \leq 0$ on $W_{R,a}$;\smallskip
 \item [ii)] $\psi(0) = 0$;\smallskip
 \item [iii)] $\psi(a) =1$;\smallskip
 \item [iv)] $\psi' >0$.\smallskip
\end{itemize}
To this aim, note that, if iv) is satisfied, then
\begin{align*}
 \Delta_{f}\psi &= \psi'' + \psi' \left\{ \frac{m}{d+R} - \langle x, \nabla r \rangle \right\}\\
 &\leq \psi'' + \psi' \left\{ \frac{m}{d+R} - |x-y| + |y - z| + |z| \right\}\\
 &= \psi'' +\psi' \left\{ \frac{m}{d+R}- d + |z| \right\}.
\end{align*}
Therefore, we are led to impose
\[
\psi'' +\psi' \left\{ \frac{m}{d+R}- d + |z| \right\} = 0.
\]
Integrating on $[0,d]\subseteq [0,a]$ and recalling ii) we get
\[
\psi(d) = R^{m}\psi'(0) \int_{0}^{d} \frac{e^\frac{t^{2}}{2}dt}{(t+R)^{m}e^{|z|t} }.
\]
This definition satisfies i), ii) and iv) provided $\psi'(0)>0$. Finally, we impose the validity of iii). This implies the choice
\begin{equation}\label{psiprime(0)}
 \psi'(0) = \frac{1} {R^{m}\mathlarger\int_{0}^{a} \frac{e^\frac{t^{2}}{2}dt}{(t+R)^{m}e^{|z|t} }}
\end{equation}
and we conclude that, the desired function $\psi$ has the expression
\begin{equation*}\label{psi}
 \psi(d) = \frac{\mathlarger\int_{0}^{d} \frac{e^\frac{t^{2}}{2}dt}{(t+R)^{m}e^{|z|t} }} {\mathlarger\int_{0}^{a} \frac{e^\frac{t^{2}}{2}dt}{(t+R)^{m}e^{|z|t} }}
\end{equation*}
We observe explicitly from \eqref{psiprime(0)} that the following  rough estimate holds
\begin{equation}\label{psiprime(0)-estimate}
0<  \psi'(0) \leq C(R,a) e^{a|z|}
\end{equation}
where
\[
 C(R,a) = \frac{(R+a)^{m}}{R^{m}a}
\]
Summarizing we have obtained that:
\begin{equation*}
\begin{cases}
 \Delta_{f}\psi(d(x)) \leq 0, & \text{in }W_{R,a}\\
 \psi = 0 = v, & \text{at }z\\
 \psi = 1 > 0 \geq v, & \text{on }\partial W_{R,a} \cap \Omega\\
 \psi \geq 0 = v, &  \text{on }\partial \Omega \cap \bar{W}_{R,a} = \GS_{2} \cap \bar{W}_{R,a}.
\end{cases}
\end{equation*}
and, similarly,
\begin{equation*}
\begin{cases}
 \Delta_{f}(-\psi(d(x)))\geq 0, & \text{in }W_{R,a}\\
(- \psi) = 0 = v, & \text{at }z\\
( -\psi) = -1 \leq v, & \text{on }\partial W_{R,a} \cap \Omega\\
(- \psi) \leq 0 = v, &  \text{on }\partial \Omega \cap \bar{W}_{R,a} = \GS_{2} \cap \bar{W}_{R,a}.
\end{cases}
\end{equation*}
In view of \eqref{dir-probl-v}, \eqref{psiprime(0)-estimate} and the fact that $\frac{\partial v}{\partial \nu} = |\nabla u|$ on $\GS_{2}$, arguing as in \cite{Tr} we conclude that
\[
 |\nabla u|(z) \leq \psi'(0)\leq C(R,a) e^{a|z|}.
\]
This latter implies the validity of \eqref{uniformgradest} by choosing $a = \min(1,\dist_{\rr^{m+1}}(z,\GS_{1})-\epsilon)$ and letting $\epsilon \searrow 0$.
\end{proof}

\subsection{Proof of Lemma \ref{lemma-fharmfunction}}
Recall from Section \ref{section-uniqueness} and Section \ref{section-existencesolution} that, in any complete weighted manifold $M^{m+1}_{f}$ with $\ric_{f} \geq K^{2}>0$, the Dirichlet problem \eqref{DirichletProblem} has a unique solution $u \in \C^{\infty}(\bar \GO)$ satisfying $0 \leq u \leq 1$. Assume now that $M^{m+1}_{f}= \rr^{m+1}_{f}$ is the Gaussian soliton and that $\GS^{m}_{1},\GS^{m}_{2} \hookrightarrow \rr^{m+1}_{f}$ are $f$-minimal. Then, by Lemma \ref{lemma-caccioppoli}
\[
\int_{\GO} |\nabla u|^{2} dv_{f} \leq C_{1} \int_{\GS_{2}} |\nabla u| dv_{2;f_{2}}.
\]
On the other hand, in view of  the asymptotic distance condition  \eqref{AsymptHp} and the normal neighborhood condition \eqref{AsymptHp2}, from \eqref{uniformgradest} of Lemma \ref{lemma-boundarygradient} we know that
\[
|\nabla u (z) | \leq C_{2} e^{\tilde b |z|^{2}},\quad \text{on } \GS_{2}
\]
for some constants $C_{2}>0$ and $0< \tilde b<1/2$. Since $\GS_{2}$ is properly immersed, it has a polynomial extrinsic volume growth; see \eqref{polygrowth}. Recalling e.g. \cite[Lemma 25]{PiRi} we deduce that $|\nabla u| \in \L^{1}(\GS_{2},dv_{2;f_{2}})$ and, therefore, $|\nabla u| \in \L^{2}(\GO,dv_{f})$, as required. The proof of Lemma \ref{lemma-fharmfunction} is completed.

\section{Variational considerations and proof of Theorem \ref{th-positivedistance}}\label{section-variational}

\subsection{The finite \texorpdfstring{$f$}{f}-energy condition from a variational viewpoint}
There is (at least) an alternative way to deduce the  finiteness condition of the $f$-energy \eqref{L1-integrability} for the bounded solution of \eqref{DirichletProblem}. It relies on a variational argument that however seems to need the validity of the asymptotic distance assumption \eqref{AsymptHp} with the more demanding condition
\begin{equation}\label{AsymptHp3}
0\leq b<\frac{1}{4}.
\end{equation}
We briefly outline the argument.\smallskip

Let $\GS_{1} \cap \GS_{2} = \emptyset$ be properly embedded hypersurfaces that separate the complete ambient space $\left(M^{m+1}, \<\cdot,\cdot\>\right)$ and let $\GO$ be the enclosed region of $M$ so that $\partial \GO = \GS_{1} \cup \GS_{2}$. Consider the orthogonal projection
\begin{equation*}
\Pi_{1} : \bar \GO \to \GS_{1}
\end{equation*}
and note that
\begin{equation*}
\dist_{M}(z,\Pi_{1}(z)) = \dist_{M}(z,\GS_{1}). 
\end{equation*}
Now, let $\psi: \rr \to [0,1]$ be the Lipschitz function
\begin{equation}\label{defpsi}
\psi(t) =
\begin{cases}
 0 & t\leq 0\\
 t & 0 <t<1\\
 1 & t \geq 1.
\end{cases}
\end{equation}
Define a locally Lipschitz function $\Psi : \bar \GO \to [0,1]$ by
\begin{equation}\label{barrier}
 \Psi(z) = \psi \left( \frac{\dist_{M}(z,\GS_{1})}{\dist_{M}(\Pi_{1}(z),\GS_{2})} \right).
\end{equation}
Obviously,
\begin{equation}\label{Psi0}
\Psi \equiv 0\, \text{ on } \GS_{1}. 
\end{equation}
Morever, since
\begin{equation*}
 \dist_{M}(z,\GS_{1}) = \dist_{M}(\Pi_{1}(z), z) \geq  \dist_{M}(\Pi_{1}(z),\GS_{2}),\, \forall z\in\GS_{2},
\end{equation*}
we have 
\begin{equation}\label{Psi1}
 \Psi \equiv 1 \, \text{ on }\GS_{2}.
\end{equation}
Finally,
\begin{equation*}
\Lip[\Psi](z)\leq C\frac{1+ \Lip[\Pi_{1}](z)}{ \dist_{M}(\Pi_{1}(z),\GS_{2})},
\end{equation*}
for some constant $C>0$. This follows from the fact that distances are globally Lipschitz functions and, furthermore,
\begin{equation*}
\frac{\dist_{M}(z,\GS_{1})}{\dist_{M}(\Pi_{1}(z),\GS_{2})} > 1 \Rightarrow \Lip[\Psi](z) = 0.
\end{equation*}
Thus, if we now specify the situation to properly embedded self-shrinkers in  $\rr^{m+1}_{f}$ and we assume the validity of the asymptotic distance condition \eqref{AsymptHp} with $b$ satisfying \eqref{AsymptHp3}, then
\[
(a)\, \Psi\in \W^{1,2}(\Omega,dv_{f}),\quad (b)\, \Psi\equiv0 \text{ on }\Sigma_{1},\quad (c)\, \Psi\equiv1 \text{ on }\Sigma_{2}.
\]

Note now that the solutions $u_{k}$ to \eqref{MixedBndPb} over the exhaustion $\GO_{k} \nearrow \GO$ and constructed using Lieberman approach coincide with those obtained by applying the direct calculus of variations to the weighted energy functional
\[
E_{k,f}(v) = \tfrac{1}{2} \int_{\GO_{k}}|\nabla v|^{2}dv_{f}
\]
on the closed convex space
\[
\W^{1,2}_{\mathcal{D}}(\bar \GO_{k},dv_{f})=\left\{v\in \W^{1,2}(\GO_{k},dv_{f}):\,\left.v\right|_{\Sigma_{1}}\equiv0,\,\mathrm{and}\,\left.v\right|_{\Sigma_{2}}\equiv1\right\}.
\]
Here, Dirichlet data are understood in the trace sense. Thus, each $u_{k}$ is a minimizer of $E_{k,f}$ over $\W_{\CD}^{1,2}(\bar \GO_{k},dv_{f})$. Thanks to the global $\W^{1,2}$-regularity established in \cite[Proposition 1.2]{IPS}, this follows from \cite[Remark 1.3]{IPS} by a suitable choice of the domains $\GO$ and $D\Subset \GO$.\smallskip

With this preparation, let $\Psi_{k}=\left.\Psi\right|_{\Omega_{k}}$ be the restriction to $\Omega_{k} \nearrow \GO_{\infty} = \GO$ of the barrier function \eqref{barrier} and let $u_{k}$ be the solution of  problem \eqref{MixedBndPb}. Recall that, up to subsequences, $u_{k}$ $\C^{2}$-converges on compact subsets of $\GO$ to the bounded solution $u$ of \eqref{DirichletProblem}. Since $\Psi_{k}\in \W^{1,2,}_{\CD}(\bar \GO_{k},dv_{f})$, we deduce
\[
E_{k_{0},f}(u_{k}) \leq E_{k,f}(u_{k}) \leq E_{k,f}(\Psi_{k}) \leq E_{\infty,f}(\Psi) < +\infty
\]
for every $k_{0}<k$. Whence, recalling that $\nabla u_{k} \to \nabla u$ uniformly on compact subsets of $\bar \GO$ and  using Fatou lemma, we conclude that  $|\nabla u|\in \L^{2}(\Omega, dv_{f})$. 

\subsection{Proof of Theorem \ref{th-positivedistance}}
If on the one hand the previous arguments require a more stringent condition on the asymptotic behaviour of $\GS_{1}$ and $\GS_{2}$ and, thus, cannot be used to recover Theorem \ref{theorem-intersection}, on the other hand they suggest a way to obtain the intersection property when no extrinsic condition on $\GS_{2}$ is imposed. Indeed, note that if the two $f$-minimal hypersurfaces $\GS^{m}_{1}, \GS^{m}_{2}$ of $M^{m+1}_{f}$ are a positive distance apart (but their positive distance could be realized at infinity), then we can consider the function defined by
\[
\Psi(z)=\psi\left(\frac{\mathrm{dist}_M(z,\Sigma_1)-\mathrm{dist}_M(z,\Sigma_2)+\mathrm{dist}_M(\Sigma_1,\Sigma_2)}{2\mathrm{dist}_M(\Sigma_1,\Sigma_2)}\right),
\] 
with $\psi$  as in \eqref{defpsi}. It is easy to check that $\Psi$ is a Lipschitz function satisfying conditions \eqref{Psi0} and \eqref{Psi1}. In particular, $\Psi$ can be used as a global barrier function and reasoning as in the previous subsection we deduce that, without any assumption on the extrinsic geometry of the hypersurfaces, the bounded solution $u$ of \eqref{DirichletProblem} satisfies $|\nabla u|\in \L^{2}(\Omega, dv_{f})$. When combined with Theorem \ref{th-abstract} and Lemma \ref{lemma-fharmfunction} this 
 fact proves Theorem \ref{th-positivedistance}.

\begin{acknowledgement*}
The authors would like to thank Carlo Mantegazza and Luciano Mari for their interest in this work and useful comments on a previous version of the paper. They are also indebted to the anonymous referee for a careful reading and for several suggestions that improved the presentation of the paper. The first author is partially supported by INdAM-GNSAGA. The second and third authors acknowledge partial support by INdAM-GNAMPA.
\end{acknowledgement*}

\bibliographystyle{amsplain}
\bibliography{FrankelSS}

\providecommand{\bysame}{\leavevmode\hbox to3em{\hrulefill}\thinspace}
\providecommand{\MR}{\relax\ifhmode\unskip\space\fi MR }
\providecommand{\MRhref}[2]{%
  \href{http://www.ams.org/mathscinet-getitem?mr=#1}{#2}
}
\providecommand{\href}[2]{#2}
\begin{thebibliography}{10}

\bibitem{CaEs}
M.~P. Cavalcante and J.~M. Espinar, \emph{Halfspace type theorems for
  self-shrinkers}, Bull. Lond. Math. Soc. \textbf{48} (2016), no.~2, 242--250.

\bibitem{ChSc}
O.~Chodosh and F.~Schulze, \emph{Uniqueness of asymptotically conical tangent
  flows}, Preprint (2019), available at
  \url{https://arxiv.org/pdf/1901.06369.pdf}.

\bibitem{ChWa}
H.I. Choi and A.-N Wang, \emph{A first eigenvalue estimate for minimal
  hypersurfaces}, J. Differ. Geom. \textbf{18} (1983), 559--562.

\bibitem{CoMi}
T.~H. Colding and W.~P. Minicozzi, \emph{Generic mean curvature flow {I}:
  generic singularities}, Ann. of Math. (2) \textbf{175} (2012), no.~2,
  755--833.

\bibitem{DiXi}
Q.~Ding and Y.~L. Xin, \emph{Volume growth, eigenvalue and compactness for
  self-shrinkers}, Asian J. Math. \textbf{17} (2013), no.~3, 443--456.

\bibitem{Fr}
T.~Frankel, \emph{On the fundamental group of a compact minimal submanifold},
  Ann. of Math. (2) \textbf{83} (1966), 68--73.

\bibitem{FL}
A.~Fraser and M.~M.-C. Li, \emph{Compactness of the space of embedded minimal
  surfaces with free boundary in three-manifolds with nonnegative {R}icci
  curvature and convex boundary}, J. Differential Geom. \textbf{96} (2014),
  no.~2, 183--200.

\bibitem{GT}
D.~Gilbarg and N.~S. Trudinger, \emph{Elliptic partial differential equations
  of second order}, Classics in Mathematics, Springer-Verlag, Berlin, 2001,
  Reprint of the 1998 edition.

\bibitem{GP}
V.~Guillemin and A.~Pollack, \emph{Differential topology}, Prentice-Hall, Inc.,
  Englewood Cliffs, N.J., 1974.

\bibitem{HoMe}
D.~Hoffman and W.~H. Meeks, \emph{The strong halfspace theorem for minimal
  surfaces}, Invent. Math. \textbf{101} (1990), 373--377.

\bibitem{ILPS}
D.~Impera, J.H. de~Lira, S.~Pigola, and A.~G. Setti, \emph{Height estimates for
  {K}illing graphs}, J. Geom. Anal. (2017), 1--29, Online first. DOI:
  10.1007/s12220-017-9938-5.

\bibitem{IPS}
D.~Impera, S.~Pigola, and A.~G. Setti, \emph{Potential theory for manifolds
  with boundary and applications to controlled mean curvature graphs}, J. Reine
  Angew. Math. \textbf{733} (2017), 121--159.

\bibitem{KKM}
N.~Kapouleas, S.~J. Kleene, and N.~M. M{\o }ller, \emph{Mean curvature
  self-shrinkers of high genus: non-compact examples}, J. Reine Angew. Math.
  \textbf{739} (2018), 1--39.

\bibitem{KlMo}
S.~J. Kleene and N.~M. M{\o}ller, \emph{Self-shrinkers with a rotational
  symmetry}, Trans. Amer. Math. Soc. \textbf{366} (2014), no.~8, 3943--3963.

\bibitem{Li1}
G.~M. Lieberman, \emph{Mixed boundary value problems for elliptic and parabolic
  differential equations of second order}, J. Math. Anal. Appl. \textbf{113}
  (1986), no.~2, 422--440.

\bibitem{MaDu}
L.~Ma and S.-H. Du, \emph{Extension of {R}eilly formula with applications to
  eigenvalue estimates for drifting {L}aplacians}, C. R. Acad. Sci., S\`er. 1
  Math. \textbf{348} (2010), 1203--1206.

\bibitem{PePiSe}
L.~F. Pessoa, S.~Pigola, and A.~G. Setti, \emph{Dirichlet parabolicity and
  {$L^1$}-{L}iouville property under localized geometric conditions}, J. Funct.
  Anal. \textbf{273} (2017), no.~2, 652--693.

\bibitem{PeWi}
P.~Petersen and F.~Wilhelm, \emph{On {F}rankel's theorem}, Canad. Math. Bull.
  \textbf{46} (2003), no.~1, 130--139.

\bibitem{PiRiSe}
S.~Pigola, M.~Rigoli, and A.~G. Setti, \emph{Some applications of integral
  formulas in {R}iemannian geometry and {PDE}'s}, Milan J. Math. \textbf{71}
  (2003), 219--281.

\bibitem{PiRi}
S.~Pigola and M.~Rimoldi, \emph{Complete self-shrinkers confined into some
  regions of the space}, Ann. Global Anal. Geom. \textbf{45} (2014), no.~1,
  47--65.

\bibitem{Re}
R.C Reilly, \emph{Applications of {H}essian operator in a {R}iemannian
  manifold}, Indiana Univ. Math. J. \textbf{26} (1977), 459--472.

\bibitem{Ri1}
M.~Rimoldi, \emph{On a classification theorem for self-shrinkers}, Proc. Amer.
  Math. Soc. \textbf{142} (2014), no.~10, 3605--3613.

\bibitem{Tr}
N.~S. Trudinger, \emph{The boundary gradient estimate for quasilinear elliptic
  and parabolic differential equations}, Indiana Univ. Math. J. \textbf{21}
  (1971/1972), 657--670.

\bibitem{Wa-preprint}
L.~Wang, \emph{Asymptotic structure of self-shrinkers}, Preprint (2016),
  available at \url{https://arxiv.org/abs/1610.04904}.

\bibitem{Wa-JAMS}
\bysame, \emph{Uniqueness of self-similar shrinkers with asymptotically conical
  ends}, J. Amer. Math. Soc. \textbf{27} (2014), no.~3, 613--638.

\bibitem{WeWy}
G.~Wei and W.~Wylie, \emph{Comparison geometry for the {B}akry-{E}mery {R}icci
  tensor}, J. Differential Geom. \textbf{83} (2009), no.~2, 377--405.

\end{thebibliography}

\end{document}